\documentclass[12pt, reqno]{amsart}
\UseRawInputEncoding
\usepackage{amsmath, amsthm, amscd, amsfonts, amssymb, graphicx, color}
\usepackage[bookmarksnumbered, colorlinks, plainpages]{hyperref}
\input{mathrsfs.sty}

\hypersetup{colorlinks=true,linkcolor=red, anchorcolor=green,
citecolor=cyan, urlcolor=red, filecolor=magenta, pdftoolbar=true}

\newtheorem{theorem}{Theorem}[section]
\newtheorem{lemma}[theorem]{Lemma}
\newtheorem{proposition}[theorem]{Proposition}
\newtheorem{corollary}[theorem]{Corollary}
\theoremstyle{definition}

\theoremstyle{remark}
\newtheorem{remark}[theorem]{Remark}
\numberwithin{equation}{section}

\newcommand{\fii}{\varphi}

\usepackage{color}
\begin{document}

\setcounter{page}{1}

\title[From norm derivatives to orthogonalities]
{From norm derivatives to orthogonalities in Hilbert $C^*$-modules}
\author[P. W\'ojcik \MakeLowercase{and} A. Zamani]
{Pawe\l{} W\'ojcik$^1$ \MakeLowercase{and} Ali Zamani$^{2,*}$}

\address{$^1$Institute of Mathematics, Pedagogical University of Cracow, Podchor\c a\.zych~2, 30-084 Krak\'ow, Poland}
\email{pawel.wojcik@up.krakow.pl}
\address{$^*$ Corresponding author, $^2$School of Mathematics and Computer Sciences, Damghan University, Damghan, P.~O.~BOX 36715-364, Iran}
\email{zamani.ali85@yahoo.com}
\begin{abstract}
Let $\big(\mathscr{X}, \langle\cdot, \cdot\rangle\big)$ be a Hilbert $C^*$-module over a $C^*$-algebra $\mathscr{A}$
and let $\mathcal{S}(\mathscr{A})$ be the set of states on $\mathscr{A}$.
In this paper, we first compute the norm derivative for elements
$x$ and $y$ of $\mathscr{X}$ as follows
\begin{align*}
\rho_{_{+}}(x, y) = \max\Big\{\mbox{Re}\,\varphi(\langle x, y\rangle): \, \varphi \in \mathcal{S}(\mathscr{A}),
\varphi(\langle x, x\rangle) = \|x\|^2\Big\}.
\end{align*}
We then apply it to characterize different concepts of orthogonality in $\mathscr{X}$.
In particular, we present a simpler proof of the classical
characterization of Birkhoff--James orthogonality in Hilbert $C^*$-modules.
Moreover, some generalized Daugavet equation in the $C^*$-algebra $\mathbb{B}(\mathcal{H})$ of all bounded
linear operators acting on a Hilbert space $\mathcal{H}$ is solved.
\end{abstract} \maketitle
\section{Introduction and Preliminaries}
Let $\big(X, \|\!\cdot\!\|\big)$ be a normed space and $X^*$ its dual topologic space.
We define two mappings $\rho_{_{+}}, \rho_{_{-}}:X\times X\rightarrow\mathbb{R}$ by the formulas
\begin{align*}
\rho_{_{\pm}}(x,y):=\lim_{t\rightarrow0^{\pm}}\frac{\|x+ty\|^2-\|x\|^2}{2t}=\|x\|\!\cdot\!\lim_{t\rightarrow0^{\pm}}\frac{\|x+ty\|-\|x\|}{t}.
\end{align*}
The convexity of the norm yields that the above definitions are meaningful. These mappings
are called the \textit{norm derivatives} and their following useful
properties can be found, e.g. in \cite{A.S.T, Dra}. For every $x$ and $y$ in $X$ and
for every $\alpha = |\alpha|e^{i\theta}$, $\beta = |\beta|e^{i\omega}$ in $\mathbb{C}$, we have
\begin{itemize}
\item[(P1)] $\rho_{_{-}}(x, y)\leq \rho_{_{+}}(x, y)$,
$|\rho_{_{\pm}}(x,y)| \leq \|x\|\!\cdot\!\|y\|$ and $\rho_{_{\pm}}(x,x) = \|x\|^2$,
\item[(P2)]  $\rho_{_{\pm}}(-x,y) = \rho_{_{\pm}}(x,-y)=-\rho_{\mp}(x,y)$,
\item[(P3)]   $\rho_{_{\pm}}(x,\alpha x + y) = \mbox{Re}\,\alpha \,\|x\|^2 + \rho_{_{\pm}}(x,y)$,
\item[(P4)] $\rho_{_{\pm}}(\alpha x,\beta y) = |\alpha \beta|\rho_{_{\pm}}(x,e^{i(\omega - \theta)}y)$,
\item[(P5)] $\rho_{_+}(x,y)=\lim\limits_{t\rightarrow 0^+}\rho_{_+}(x+ty,y)$.
\end{itemize}
If the norm on $X$ comes from an inner product $[\cdot, \cdot]$, then we obtain
$\rho_{_{+}}(x,y) = \rho_{_{-}}(x,y) = \mbox{Re}\,[x, y]$ for all $x, y\in X$,
i.e. both $\rho_{_{+}}$ and $\rho_{_{-}}$ are perfect generalizations of inner products.

In more general cases mappings $\rho_{_{+}}$ and $\rho_{_{-}}$ are useful for applications in approximation
theory and, in particular, they played a significant role in the paper \cite{Wo.2019}.
For more information about the norm derivatives and their
properties the reader is referred to \cite{A.S.T, Dra} (see also \cite{C.W.1, C.W.2, M.Z.D, Z.M}).

From \cite{Dra}, for two elements $x$ and $y$ of a normed linear space $X$, we have
\begin{align}\label{rho-plus-max-property}
\rho_{_{+}}(x, y) = \|x\|\max\big\{\mbox{Re}\,x^*(y): \, x^*\in J(x)\big\},
\end{align}
where $J(x): = \big\{x^*\in X^*:\,\|x^*\| = 1,\, x^*(x) = \|x\|\big\}$.
If we have additional structures on a normed linear space $X$,
then we obtain other expressions for the norm derivative $\rho_{_{+}}$.
Therefore we survey on the well known results involving norm derivatives.
So the present section has an expository character in part;
however, many of the surveyed results will be essentially extended in the next section.

In the classical Banach space $\mathcal{C}(K)$ of all continuous functions on a
compact Hausdroff space $K$, the result for $f, g\in \mathcal{C}(K)$ were given by Kecki\'{c} \cite{Ke.3}:
\begin{align*}
\rho_{_{+}}(f, g) = \|f\|\max \Big\{\mbox{Re}\,\big(e^{-i\arg (f(x))}g(x)\big): \, x\in M_f\Big\}
\end{align*}
where $M_f:= \{x\in K: \, |f(x)| = \|f\|\}$.

In $\mathbb{B}(\mathcal{H})$, the $C^*$-algebra of all
bounded linear operators on a Hilbert space $\big(\mathcal{H}, [\cdot, \cdot]\big)$, and for $T, S\in \mathbb{B}(\mathcal{H})$,
Kecki\'{c} \cite{Ke.2} obtained the following formula:
\begin{align*}
\rho_{_{+}}(T, S) = \inf_{\varepsilon>0}\sup\Big\{\mbox{Re}[Sx, Tx]: \,x\in {\mathcal{H}}_{\varepsilon}, \|x\| =1\Big\},
\end{align*}
where ${\mathcal{H}}_{\varepsilon} := E_{T^*T}\big((\|T\| - \varepsilon)^2, \|T\|^2\big)$,
and $E_{T^*T}$ stands for the spectral measure of the operator $T^*T$.

Norm derivatives of the space $\mathbb{K}(\mathcal{H})$ (compact operators on $\mathcal{H}$)
has been studied in \cite{Ke.1}.
More precisely, for $T, S \in \mathbb{K}(\mathcal{H})$ where $T = U|T|$ is the polar decomposition of $T$, we have
\begin{align*}
\rho_{_{+}}(T, S) = \|T\|\max\Big\{\mbox{Re}[U^*Sx, x];\,\, x\in \Phi, \, \|x\| = 1\Big\},
\end{align*}
where $\Phi$ is the characteristic subspace of $T$ with respect to its eigenvalue $s_1$.

W\'{o}jcik \cite{Wo.2}, by using a different method, extended this result for
compact operators between real normed spaces, i.e. $\mathbb{K}(X,Y)$.
Moreover, similar investigations have been carried out by W\'{o}jcik \cite{Wo.2017} in $M$-ideals in bounded
operator space $\mathbb{B}(X,Y)$. The main result of \cite{Wo.2017}
says that if $T, S\in \mathbb{B}(X,Y)$ and ${\rm dist}\big(T,\mathbb{K}(X,Y)\big)<\|T\|$, then
\begin{align*}
\rho_{_{+}}(T, S) = \|T\|\max\Big\{\rho_{_{+}}(Tx,Sx):\,\, x\in {\rm Ext}B_X, \, \|Tx\| = \|T\|\Big\},
\end{align*}
where ${\rm Ext}B_X$ denotes the set of all extremal points of the closed unit ball $B_X$.

Motivated by the above properties, we compute the norm derivatives in Hilbert $C^*$-modules.
Namely, for two elements $x$ and $y$ of a Hilbert $\mathscr{A}$-module $\big(\mathscr{X}, \langle\cdot, \cdot\rangle\big)$
we will prove that
\begin{align*}
\rho_{_{+}}(x, y) = \max\Big\{\mbox{Re}\,\varphi(\langle x, y\rangle):
\, \varphi \in \mathcal{S}(\mathscr{A}),\, \varphi(\langle x, x\rangle) = \|x\|^2\Big\},
\end{align*}
where $\mathcal{S}(\mathscr{A})$ is the set of states on $\mathscr{A}$.
This formula enables us to characterize different concepts of Birkhoff–-James orthogonality
for elements of a Hilbert $C^*$-module. Some other related results are also discussed.

Before stating the results, we establish the notation and recall some definitions from the literature.
An element $a$ in a \textit{$C^*$-algebra} $\mathscr{A}$ is called \textit{positive} (we write $a\geq0$)
if $a = b^*b$ for some $b\in \mathscr{A}$. A linear functional $\varphi$ of $\mathscr{A}$ is \textit{positive}
if $\varphi(a)\geq0$ for every positive element
$a\in \mathscr{A}$. A \textit{state} is a positive linear functional whose norm is equal to one.
The symbol $\mathcal{S}(\mathscr{A})$ denotes the set of states on $\mathscr{A}$.

An \textit{inner product module} over $\mathscr{A}$ is a (left)
$\mathscr{A}$-module $\mathscr{X}$ equipped with an \textit{$\mathscr{A}$-valued
inner product} $\langle\cdot, \cdot\rangle$, which is $\mathbb{C}$-linear and
$\mathscr{A}$-linear in the first variable and has the properties $\langle x, y
\rangle^*=\langle y, x\rangle$ as well as $\langle x, x\rangle \geq 0$ with equality
if and only if $x = 0$.
An inner product $\mathscr{A}$-module $\mathscr{X}$ is called a \textit{Hilbert $\mathscr{A}$-module}
if it is complete with respect to the norm $\|x\| = {\|\langle x, x\rangle\|}^{\frac{1}{2}}$.
For $x\in\mathscr{X}$, by \cite[Theorem~3.3.6]{Mu}, there always
exists a $\varphi \in \mathcal{S}(\mathscr{A})$ such that $\varphi\big(\langle x, x\rangle\big) = \|x\|^2$.
So, let $\Omega_x$ denote the (nonempty) subset of the set of supporting functionals:
\begin{align*}
\Omega_x:=\big\{\varphi \in \mathcal{S}(\mathscr{A}): \varphi\big(\langle x, x\rangle\big) = \|x\|^2\big\}\subseteq J(x).
\end{align*}
Given a positive functional $\varphi$ on $\mathscr{A}$, we have the following useful version of the Cauchy--Schwarz inequality:
\begin{align}\label{inequality-positive-functional}
|\varphi(\langle x, y\rangle)|^2\leq \varphi(\langle x, x\rangle)\varphi(\langle y, y\rangle) \quad (x, y\in \mathscr{X}).
\end{align}
Every $C^*$-algebra $\mathscr{A}$ can be regarded as a Hilbert $C^*$-module over itself where the inner product is
defined by $\langle a, b\rangle := a^*b$.
By $\mathbb{M}_{n}(\mathbb{C})$ we denote the $C^*$-algebra of all
complex $n\times n$ matrices. We shall identify $\mathbb{B}(\mathbb{C}^n)$
and $\mathbb{M}_{n}(\mathbb{C})$ in the usual way.
We refer the reader to \cite{Dix, M.T} for more information on Hilbert $C^*$-modules.

A concept of orthogonality in a Hilbert $\mathscr{A}$-module $\mathscr{X}$ can be defined with respect to the
$\mathscr{A}$-valued inner product in a natural way, that is, two elements $x$ and $y$ of $\mathscr{X}$
are orthogonal, in short $x \perp y$, if $\langle x, y\rangle = 0$.
There are many different ways how one can extend this notion, see \cite{A.R.2} and the references therein.
One of them is the Birkhoff–-James orthogonality:
we say that $x$ and $y$ are Birkhoff–-James orthogonal, and we write $x\perp_B y$, if
$\|x\| \leq \|x + \lambda y\|$ for all $\lambda\in\mathbb{C}$.
A well known characterization of the Birkhoff--James orthogonality is
due to Aramba\v{s}i\'{c} and Raji\'{c} (see \cite{A.R.1}):
\begin{align}\label{a-r-b-g--theo-bj-orth}
x\perp_B y \quad\Leftrightarrow\quad \big(\exists \varphi\in\Omega_x : \varphi(\langle x,y\rangle)=0\big).
\end{align}
It has been proved also by Bhattacharyya and Grover (cf. \cite{B.G}).
The statement of this nice characterization is so simple, and its existing proofs so extremely
long, that one is easily seduced into an effort to find a simpler, shorter proof. The
present work is the result of our attempt.
It is worth mentioning that those four mathematicians applied a faithful
representation $\pi\colon \mathscr{A}\to \mathbb{B}(\mathcal{H})$
(see \cite[Theorem 2.6.1]{Dix}) and linking algebra of $\mathscr{X}$ (see \cite{M.T}).
But we do not use this strong tool. We apply norm derivatives.
In this paper we would like to present a simpler proof of this nice result and we
demonstrate the power of the norm derivatives. We hope that it sheds new light on this
intricate geometric structure of Hilbert $C^*$-modules and will provide the great applications of the
mappings $\rho_{_{\pm}}$ in the future.
\section{Main results: Norm derivatives in Hilbert $C^*$-modules}
In this section, we first compute the norm derivatives in Hilbert $C^*$-modules.
Then, as an application of our results, we get an explicit formula for
the norm derivatives $\rho_{_{\pm}}$ of certain elements in Hilbert $C^*$-modules.
Moreover, we apply our results to give some solutions of the
generalized Daugavet equation in the operator space $\mathbb{B}(\mathcal{H})$.

We start our work with the following lemma.
\begin{lemma}{\rm \cite{Wo.2012}}\label{lem-star-shaped-subset}
Suppose that $X$ is a real normed space.
Let $D\subseteq X$ be a dense and star-shaped subset (i.e. $\alpha D\subseteq D$ for all $\alpha>0$). Let
$M$ be a closed affine hyperplane (i.e. {\rm codim}M=1) such that $0\notin M$. Then ${\rm cl}(M\cap D)=M$.
\end{lemma}
It is worth mentioning that Lemma \ref{lem-star-shaped-subset}
played a significant role in the papers \cite{Wo.2012} and \cite{Wo.2015}.
Its proof can be found in \cite{Wo.2012}. Here, this lemma will be a helpful tool, again.

We say that $X$ is \textit{smooth at point} $x_o$ if there is a unique functional $f\in J(x_o)$.
Now, we consider a set $\mathcal{N}_{sm}(X):=\{x\in X : X\ {\rm is\ smooth\ at}\ x\}\cup\{0\}$.
In particular, we have
\begin{align}\label{smoothness-rho-plus-minus-f}
x_o\in\mathcal{N}_{sm}(X)\quad\Leftrightarrow\quad\rho_{_+}(x_o,\cdot)=\rho_-(x_o,\cdot)=\|x\|{\rm Re}\,f(\cdot), \ \ f\in J(x_o).
\end{align}
It is known that if $\dim X<\infty$, then $\mathcal{N}_{sm}(X)$ is a
dense, star-shaped subset of $X$ -- cf. \cite{Wo.2012} or \cite{A.S.T}. Thus we
can rewrite Lemma \ref{lem-star-shaped-subset} as
\begin{lemma}\label{lem-star-shaped-subset-sm}
Suppose that $Z$ is a two-dimensional real normed space. Let $\{x,y\}\subseteq Z$ be a linearly independent subset.
If we consider a line $M$ spanned by
the vectors $x,y$ (i.e. $M:=\big\{x+ty\in Z: t\in \mathbb{R}\big\}$), then ${\rm cl}(M\cap \mathcal{N}_{sm}(Z))=M$.
\end{lemma}
We are now in a position to prove the main result of this paper.
\begin{theorem}\label{T.3.2}
Let $\mathscr{X}$ be a Hilbert $\mathscr{A}$-module, and $x, y \in\mathscr{X}$. Then
\begin{align*}
\rho_{_{+}}(x, y) = \max\Big\{{\rm Re}\,\varphi(\langle x, y\rangle): \, \varphi \in \Omega_x\Big\}.
\end{align*}
\end{theorem}
\begin{proof}
We may and shall assume that $x\neq0$ otherwise the statement trivially holds.
Let $\varphi\in\Omega_x$, that is, $\varphi \in \mathcal{S}(\mathscr{A})$ and $\varphi(\langle x, x\rangle) = \|x\|^2$.
It is easy to check that $\frac{1}{\|x\|}\varphi(\langle x, \cdot\rangle)\in J(x)$. Thus the
property \eqref{rho-plus-max-property} yields
$\|x\|\mbox{Re}\,\frac{1}{\|x\|}\varphi(\langle x, y\rangle)\leq\rho'_+(x,y)$.
Hence $\mbox{Re}\,\varphi(\langle x,y\rangle)\leq\rho_{_+}(x,y)$.
Passing to the supremum over $\fii\in\Omega_x$ we get
\begin{center}
$\sup\Big\{\mbox{Re}\,\varphi(\langle x, y\rangle): \, \varphi \in \Omega_x\Big\}\leq\rho_{_+}(x,y)$.
\end{center}
To complete the proof, we must find $\fii \in \Omega_x$
such that $\mbox{Re}\,\fii(\langle x, y\rangle) = \rho_{_{+}}(x, y)$. Then we will able to
write "max" instead of "sup".
We consider two cases. First assume that $x$ and $y$ are linearly dependent, i.e. $y=\alpha x$ with
some number $\alpha$. Fix $\fii \in \Omega_x$. Then
\begin{align*}
\rho_{_+}(x,y)&=\rho_{_+}(x,\alpha x)\stackrel{{\rm(P3)}}{=}{\rm Re}\big(\alpha\|x\|^2\big)={\rm Re}\big(\alpha\fii(\langle x, x\rangle)\big)\\
&={\rm Re}\,\fii(\langle x, \alpha x\rangle)={\rm Re}\,\fii(\langle x, y\rangle).
\end{align*}
So, the first case is complete. Now, suppose that $\{x,y\}$ is a linearly independent subset.
Let us define a real subspace $Z\subseteq\mathscr{X}$ by the formula
\begin{center}
$Z:=\big\{\alpha x +\beta y\in \mathscr{X}: \alpha,\beta\in \mathbb{R} \big\}$.
\end{center}
Put $M:=\big\{x+ty\in Z: t\in \mathbb{R}\big\}$. It follows from
Lemma \ref{lem-star-shaped-subset-sm} that ${\rm cl}(M\cap \mathcal{N}_{sm}(Z))=M$. This equality yields, in
particular, that there exists a sequence $\{t_n:n=1,2,\ldots\}$ in $(0,+\infty)$ such that $t_n\rightarrow 0^+$ and
$x+t_ny\in \mathcal{N}_{sm}(Z)$.

Let us define for the moment a closed subalgebra $\mathscr{A}_o\subseteq\mathscr{A}$
spanned by the vectors
$\langle x,x \rangle$, $\langle y,y \rangle$, $\langle x,y \rangle$ and $\langle y,x \rangle$.
In particular, we have
\begin{center}
$\langle x+t_ny, x+t_ny\rangle,\ \langle x+t_ny,y\rangle\in \mathscr{A}_o$.
\end{center}
Moreover, it is easily seen that $\mathscr{A}_o$ is separable.
Also, there exist states $\fii_1,\fii_2,\fii_3,\ldots\in\mathcal{S}(\mathscr{A}_o)$ such that
\begin{align}\label{fii-n-attains-norm2}
\fii_n(\langle x+t_ny, x+t_ny\rangle) = \|x+t_ny\|^2.
\end{align}
Now our attention must focus on the real two-dimensional space $Z$.
Fix $n\in \mathbb{N}$. It is not difficult to see that
$\frac{1}{\|x+t_ny\|}{\rm Re}\,\fii_n\big(\langle x+t_ny, \cdot\rangle \big)\big|_{_Z}\colon Z\to \mathbb{R}$
is a $\mathbb{R}$-linear functional and
\begin{align}\label{smoothness-rho-for-z}
\frac{1}{\|x+t_ny\|}{\rm Re}\,\fii_n\big(\langle x+t_ny, \cdot\rangle \big)\big|_{_Z}\in J(x+t_ny)\big|_{_Z}.
\end{align}
The space $Z$ is smooth at $x+t_ny$. Thus, from \eqref{smoothness-rho-plus-minus-f} and \eqref{smoothness-rho-for-z} it is known that
\begin{align}\label{rho-tn-fii-tn}
\rho_{_+}(x+t_ny,y)={\rm Re}\,\fii_n\big(\langle x+t_ny,y\rangle \big).
\end{align}
By Alaoglu's Theorem, we known that the closed unit ball $B_{\mathscr{A}_o^*}$ is weakly* compact.
It is easy to check that subset $\mathcal{S}(\mathscr{A}_o)\subseteq \mathscr{A}_o^*$ is a norm-closed convex
subset of the weakly* compact ball of $B_{\mathscr{A}_o^*}$. Therefore $\mathcal{S}(\mathscr{A}_o)$ is weakly* compact.
Since $\mathscr{A}_o$ is separable, $\mathcal{S}(\mathscr{A}_o)$ is weakly* sequentially compact.
Thus, there are an element $\fii_o\in B_{\mathscr{A}_o^*}$ and a
subsequence $\big(\fii_{n_k} \big)_{k=1}^{\infty}\subseteq B_{\mathscr{A}_o^*}$
such that $\fii_{n_k} \stackrel{w*}{\longrightarrow} \fii_{o}$.
Since $\langle x+t_{n_k}y,y\rangle\stackrel{\|\cdot\|}{\longrightarrow}\langle x,y\rangle$, we conclude that
\begin{align}\label{subsequence-re-fii-1}
{\rm Re}\,\fii_{n_k}\big(\langle x+t_{n_k}y,y\rangle \big)\longrightarrow {\rm Re}\fii_o(\langle x,y\rangle).
\end{align}
Similarly, since $\langle x+t_{n_k}y,x+t_{n_k}y\rangle\stackrel{\|\cdot\|}{\longrightarrow}\langle x,x\rangle$, it
follows that
\begin{align}\label{subsequence-re-fii-2}
\fii_{n_k}\big(\langle x+t_{n_k}y,x+t_{n_k}y\rangle \big)\longrightarrow \fii_o(\langle x,x\rangle).
\end{align}
Combining the three conditions (P5)-\eqref{rho-tn-fii-tn}-\eqref{subsequence-re-fii-1} we deduce that
\begin{align*}
{\rm Re}\,\fii_{o}\big(\langle x,y\rangle \big)=\rho_{_+}(x,y).
\end{align*}
Furthermore, combining \eqref{fii-n-attains-norm2} with \eqref{subsequence-re-fii-2} yields
$\fii_o(\langle x,x\rangle)=\|x\|^2$.

Now we back to the $C^*$-algebra $\mathscr{A}$ and we go to the dual normed
space $\mathscr{A}^*$. Namely, since $\fii_o\colon \mathscr{A}_o\to \mathbb{C}$ is a state on $\mathscr{A}_o$,
it follows that there exists a state $\fii$ on $\mathscr{A}$ such
that $\fii|_{_{\mathscr{A}_o}}=\fii_o$ (see e.g. \cite[p.259]{acfa_conway}). From this
we get ${\rm Re}\,\fii(\langle x,y\rangle)=\rho_{_+}(x,y)$ and
$\fii(\langle x,x\rangle)=\|x\|^2$. The proof is completed.
\end{proof}
Now we are able to calculate a formula for the norm derivative $\rho_{_{-}}$ in Hilbert $C^*$-modules.
\begin{theorem}\label{C.4.2}
Let $\mathscr{X}$ be a Hilbert $\mathscr{A}$-module, and $x, y \in\mathscr{X}$. Then
\begin{align*}
\rho_{_{-}}(x, y) = \min\Big\{{\rm Re}\,\varphi(\langle x, y\rangle): \, \varphi \in \Omega_x\Big\}.
\end{align*}
\end{theorem}
\begin{proof}
Applying (P2) and Theorem~\ref{T.3.2} we get
\begin{align*}
\rho_{_{-}}(x, y) &=-\rho_{_{+}}(x, -y)= - \max\Big\{{\rm Re}\,\varphi(\langle x, -y\rangle):  \, \varphi \in \Omega_x\Big\}
\\& = \min\Big\{{\rm Re}\,\varphi(\langle x, y\rangle): \,  \varphi \in \Omega_x\Big\},
\end{align*}
and we are done.
\end{proof}
Recall that every $C^*$-algebra $\mathscr{A}$ can be regarded as a Hilbert $C^*$-module over
itself with the inner product $\langle a, b\rangle := a^*b$.
Thus, as a consequence of Theorem~\ref{T.3.2} (and Theorem~\ref{C.4.2}) we have the following result.
\begin{corollary}\label{P.2.2}
Let $\mathscr{A}$ be a $C^*$-algebra, and $a, b\in \mathscr{A}$. Then
\begin{align*}
\rho_{_{+}}(a, b) = \max\Big\{{\rm Re}\,\varphi(a^*b): \, \varphi \in \Omega_a\Big\},\\
\rho_{_{-}}(a, b) = \min\Big\{{\rm Re}\,\varphi(a^*b): \, \varphi \in \Omega_a\Big\}.
\end{align*}
\end{corollary}
Let $\mathbb{T}(\mathcal{H})$ be trace-class operators on a Hilbert space $\big(\mathcal{H}, [\cdot, \cdot]\big)$.
It is well known that (see e.g. \cite[Theorem 4.2.1]{Mu}) every state $\varphi$ of $\mathbb{K}(\mathcal{H})$
is of the form $a \rightarrow {\rm tr}(pa)$ for some positive trace one operator $p\in\mathbb{T}(\mathcal{H})$.
Therefore, as an immediate consequence of Theorem~\ref{T.3.2} and Theorem~\ref{C.4.2}, we have the following result.
\begin{proposition}\label{P.5.2}
Let $\mathscr{X}$ be a Hilbert $\mathbb{K}(\mathcal{H})$-module, and $x, y \in\mathscr{X}$. Then
\begin{align*}
\rho_{_{+}}(x, y) = \max\Big\{{\rm Re}\,{\rm tr}\big(p\langle x, y\rangle\big):
\, p\in\mathbb{P}(\mathcal{H}),
{\rm tr}\big(p\langle x, x\rangle\big) = \|x\|^2\Big\},\\
\rho_{_{-}}(x, y) = \min\Big\{{\rm Re}\,{\rm tr}\big(p\langle x, y\rangle\big):
\, p\in\mathbb{P}(\mathcal{H}),
{\rm tr}\big(p\langle x, x\rangle\big) = \|x\|^2\Big\},
\end{align*}
where $\mathbb{P}(\mathcal{H}) = \big\{p\in\mathbb{T}(\mathcal{H}):\, p\,\mbox{is positive trace one}\,\big\}$.
\end{proposition}
We continue this section by applying our results to get an explicit formula for
the norm derivatives $\rho_{_{\pm}}$ of certain elements in Hilbert $C^*$-modules.
\begin{theorem}\label{P.10.2}
Let $\mathscr{X}$ be a Hilbert $\mathscr{A}$-module, and $x\in\mathscr{X}$. Then
\begin{align}\label{equality-rho-plus-x4-minus}
\rho_{_{+}}\big(x, x\langle x, x\rangle\big) = \|x\|^4=\rho_{_{-}}\big(x, x\langle x, x\rangle\big) .
\end{align}
\end{theorem}
\begin{proof}
Fix $\fii\in \mathcal{S}(\mathscr{A})$ such
that $\varphi(\langle x, x\rangle) = \|x\|^2$. We may suppose
that $\mathscr{A}$ is a Banach algebra with identity (by going to
extensions $\tilde{\mathscr{A}}$ and $\tilde{\varphi}=\fii\big|_{_{\mathscr{A}}}$ if
necessary; see \cite[pp.194, 259]{acfa_conway}). So, let $e$ be the identity of $\mathscr{A}$. Then
\begin{align*}
\|x\|^4 = |\varphi\big(\langle x, x\rangle e\big)|^2 \stackrel{\eqref{inequality-positive-functional}}{\leq}
\varphi\big(\langle x, x\rangle^2\big) \varphi(e^2)
\leq \varphi\big(\langle x, x\rangle^2\big) \leq \big\|\langle x, x\rangle^2\big\|\leq\|x\|^4.
\end{align*}
This implies
\begin{align}\label{I.1.P.10.2}
\varphi\big(\langle x, x\rangle^2\big) = \|x\|^4.
\end{align}
Therefore we obtain the following equalities
\begin{align*}
\mbox{Re}\,\varphi\Big(\big\langle x, \|x\|^2x - x\langle x, x\rangle\big\rangle\Big)&
= \|x\|^2\mbox{Re}\varphi(\langle x,x\rangle)- \mbox{Re}\varphi\Big(\big\langle x,x\langle x, x\rangle\big\rangle\Big)
\\& = \|x\|^2\mbox{Re}\|x\|^2- \mbox{Re}\varphi\big(\langle x,x\rangle\langle x,x\rangle\big)
\\& \stackrel{\eqref{I.1.P.10.2}}{=} \|x\|^2\|x\|^2-\|x\|^4 = 0.
\end{align*}
Now we apply Theorem  \ref{T.3.2} (resp. Theorem \ref{C.4.2}). Namely, since $\fii$ was
arbitrarily chosen from $\Omega_x$, passing to the
maximum (resp. minimum) over $\fii\in\Omega_x$, we get
\begin{center}
$\rho_{_{+}}\big(x, \|x\|^2x - x\langle x, x\rangle\big)=0$\quad and\quad $\rho_{_{-}}\big(x, \|x\|^2x - x\langle x, x\rangle\big)=0$.
\end{center}
So, from (P3) and (P2) we obtain, respectively,
$\|x\|^2\!\cdot\!\|x\|^2-\rho_{_{-}}\big(x, \langle x, x\rangle\big)=0$ and
$\|x\|^2\!\cdot\!\|x\|^2-\rho_{_{+}}\big(x,x\langle x, x\rangle\big)=0$,
and we may consider \eqref{equality-rho-plus-x4-minus} as shown.
\end{proof}
An element $x$ in a normed space $(X,\|\!\cdot\|)$
is called \textit{norm-parallel} to another element $y\in X$ (see \cite{Z.M.2015} and the references therein),
denoted by $x\parallel y$, if $\|x+\xi y\|=\|x\|+\|y\|$ for some complex unit $\xi$.
In the framework of inner product spaces, the norm-parallel relation is exactly the usual
vectorial parallel relation, that is, $x\parallel y$ if and only if $x$ and $y$ are linearly dependent. In
the setting of normed linear spaces, two linearly dependent vectors are norm-parallel, but the converse is false in general.
Many characterizations of the norm-parallelism for
operators spaces $\mathbb{B}(X,Y)$ and
elements of an arbitrary Hilbert $C^*$-module were
given in \cite{BCMWZ-2019}, \cite{Wo.2017-indag}, \cite{Z.M.2015} and \cite{ZM-2019}.
\begin{theorem}\label{theorem-daugavet-010}
Suppose that $\mathscr{X}$ is a Hilbert $\mathscr{A}$-module, and
let $x\in\mathscr{X}$, $\alpha,\beta\in(0,+\infty)$. Then
\begin{align}\label{daugavet-eq-in-mod}
\big\|\alpha x+\beta x\langle x, x\rangle\big\|=\alpha\|x\|+\beta\|x\|^3.
\end{align}
In particular, $x$ is norm-parallel to $x\langle x, x\rangle$.
\end{theorem}
\begin{proof}
It is clear that we may assume $x\neq 0$. Then we have
\begin{align*}
\alpha\|x\|^2+\beta\|x\|^4&\stackrel{\eqref{equality-rho-plus-x4-minus}}{=}\alpha\|x\|^2+\beta\rho_{_{+}}\big(x, x\langle x, x\rangle\big)\stackrel{{\rm (P3,P4)}}{=}\rho_{_{+}}\big(x,\alpha x+\beta x\langle x, x\rangle\big)\\
&\stackrel{{\rm (P1)}}{\leq}\|x\|\!\cdot\!\big\|\alpha x+\beta x\langle x, x\rangle\big\|\leq\|x\|\!\cdot\!\Big(\alpha \|x\|+\beta \big\|x\langle x, x\rangle\big\|\Big)\\
&\leq\|x\|\!\cdot\!\Big(\alpha \|x\|+\beta \|x\|\!\cdot\!\big\|\langle x, x\rangle\big\|\Big)=\alpha\|x\|^2+\beta\|x\|^4.
\end{align*}
Thus the string of inequalities becomes a string of equalities and we obtain
$\alpha\|x\|^2+\beta\|x\|^4=\|x\|\!\cdot\!\big\|\alpha x+\beta x\langle x, x\rangle\big\|$.
Dividing by $\|x\|$, we have \eqref{daugavet-eq-in-mod}.
\end{proof}
In the context of bounded linear operators on normed spaces, the well-known \textit{Daugavet equation}
$\|I+T\|= 1+\|T\|$ is a particular case of parallelism.
We refer to the book \cite{Werner-book-1996} and more recent the paper \cite{Wo.2017-sm} for motivations,
history, various aspects and problems connected with the Daugavet equation. It is worth mentioning that a generalized
Daugavet equation $\|T+S\|=\|T\|+\|S\|$ is one useful property in solving a variety of problems in approximation
theory (cf. \cite{Wo.2017-sm}).

A consequence of Theorem \ref{theorem-daugavet-010} is established in the next result.
\begin{corollary}\label{theorem-a-aaa-a-a3}
Let $T\in \mathbb{B}(\mathcal{H})$. Then
\begin{align*}
\|T+TT^*T\|=\|T\|+\|T\|^3.
\end{align*}
Moreover, if ${\rm dist}(T,\mathbb{K}(\mathcal{H}))<\|T\|$
or ${\rm dist}(TT^*T,\mathbb{K}(\mathcal{H}))<\|TT^*T\|$, then
there exists a unit vector $x_o\in \mathcal{H}$ such
that $\frac{T}{\|T\|}x_o=\frac{TT^*T}{\|TT^*T\|}x_o$, $\|Tx_o\|=\|T\|$ and $\|TT^*Tx_o\|=\|TT^*T\|$.
\end{corollary}
\begin{proof}
It follows from \eqref{daugavet-eq-in-mod} that $\|T+TT^*T\|=\|T\|+\|T\|^3$.
In addition, it is not difficult to check that $\|TT^*T\|=\|T\|^3$.
Therefore both $T$ and $TT^*T$ satisfy the generalized Daugavet equation
\begin{align}\label{a-plus-aaa--a-plus-aaa}
\|T+TT^*T\|=\|T\|+\|TT^*T\|.
\end{align}
Now, suppose that ${\rm dist}(T,\mathbb{K}(\mathcal{H}))<\|T\|$
or ${\rm dist}(TT^*T,\mathbb{K}(\mathcal{H}))<\|TT^*T\|$. Recall
that $\mathbb{K}(\mathcal{H})$ is an $M$-ideal in $\mathbb{B}(\mathcal{H})$ (see \cite{Wo.2017-indag}).
Moreover, $\mathcal{H}$ is strictly convex. Thus these information are legitimates to apply \cite[Theorem~4.4]{Wo.2017-sm}.
So, it follows from \cite[Theorem~4.4]{Wo.2017-sm} and \eqref{a-plus-aaa--a-plus-aaa} that there
exists a unit vector $x_o\in\mathcal{H}$ such that $\frac{T}{\|T\|}x_o=\frac{TT^*T}{\|TT^*T\|}x_o$,
$\|Tx_o\|=\|T\|$ and $\|TT^*Tx_o\|=\|TT^*T\|$.
\end{proof}
\section{An application: Orthogonalities in Hilbert $C^*$-modules}

In a normed linear space $\big(X,\|\!\cdot\!\|)$, for two vectors $x, y\in X$, one can
consider the \textit{Birkhoff--James orthogonality} (see \cite{B, J}) defined by
\begin{align*}
x\perp_B y\quad :\Leftrightarrow\quad \forall_{\lambda\in\mathbb{C}}\ \|x\|\leq \|x + \lambda y\|.
\end{align*}
We will consider also the \textit{r-Birkhoff--James orthogonality}, defined by
\begin{align*}
x\perp^r_B y\quad :\Leftrightarrow\quad \forall_{\alpha\in\mathbb{R}}\ \|x\|\leq \|x + \alpha y\|.
\end{align*}
Notice that $\perp_B \,\subseteq \,\perp^r_B$, but the converse is false in general.
For example, let us take $X = \mathbb{C}^2$ and let $x=(1,0)$, $y=(i,0)$. Then for all $\alpha\in\mathbb{R}$ we have
\begin{align*}
\|x + \alpha y\| = \|(1+\alpha i,0)\| = \sqrt{1 + \alpha^2} \geq 1 = \|x\|.
\end{align*}
Hence $x\perp^r_B y$. But $x\not\perp_B y$ since for $\lambda := i$ we have $\|x + \lambda y\| = 0 < 1 = \|x\|$.

We will use the following characterizations of both orthogonality relations. Namely, for arbitrary $x, y \in X$,
we have (see \cite{A.S.T, Ke.1}):
\begin{align}\label{I.151.P.2.2}
x\perp_B y \quad\Leftrightarrow \quad \displaystyle{\inf_{0\leq \theta < 2\pi}}\,\rho_{_{+}}(x, e^{i\theta}y) \geq 0
\end{align}
and
\begin{align}\label{I.150.P.2.2}
x\perp^r_B y \quad \Leftrightarrow\quad \rho_{_{-}}(x, y)\leq 0\leq\rho_{_{+}}(x, y).
\end{align}
When $X = \mathbb{M}_{n}(\mathbb{C})$ and $T, S \in X$,
a very tractable condition of the Birkhoff–-James orthogonality was found by Bhatia and \v{S}emrl in \cite{B.S}.
They showed that $T\perp_B S$ if and only if there exists a unit vector $x\in \mathbb{C}^n$ such that
\begin{align}\label{condition-b-s-for-orth}
\|Tx\| = \|T\| \quad \mbox{and} \quad [Tx, Sx] = 0.
\end{align}
Later Bhattacharyya and Grover \cite{B.G} showed that $T\perp^r_B S$
if and only if there exists a unit vector $x\in \mathbb{C}^n$ such that
\begin{align}\label{condition-b-g-for-orth}
\|Tx\| = \|T\| \quad \mbox{and} \quad \mbox{Re}\,[Tx, Sx] = 0.
\end{align}
To summarize, the papers \cite{B.S}, \cite{B.G} and
conditions \eqref{condition-b-s-for-orth}, \eqref{condition-b-g-for-orth} motivate the next theorem. In other words,
we now obtain a characterization of real version of the Birkhoff--James orthogonality in Hilbert $C^*$-modules
in terms of states of the underlying $C^*$-algebra.
\begin{theorem}\label{T.777.2}
Let $\mathscr{X}$ be a Hilbert $\mathscr{A}$-module, and $x, y \in\mathscr{X}$.
The following statements are equivalent:
\begin{itemize}
\item[(i)] $x\perp^r_B y$,
\item[(ii)] there exists $\varphi \in\Omega_x$ such that ${\rm Re}\,\varphi(\langle x, y\rangle) = 0$.
\end{itemize}
\end{theorem}
\begin{proof}
We assume that $x\neq 0$ (otherwise the result is trivial). First, suppose that $x\perp^r_B y$.
By \eqref{I.150.P.2.2} we have $\rho_{_{-}}(x, y)\leq 0\leq\rho_{_{+}}(x, y)$.
So, by Theorem \ref{T.3.2}, there exist $\fii_1,\fii_2\in\Omega_x$ such that
\begin{align*}
{\rm Re}\fii_1(\langle x,y\rangle)\leq 0 \leq {\rm Re}\fii_2(\langle x,y\rangle).
\end{align*}
It follows from the above inequality that for some $\lambda_o\in [0, 1]$ we have
\begin{align}\label{two-states-rho-o}
\lambda_o{\rm Re}\fii_1\big(\langle x,y\rangle\big)+(1-\lambda_o){\rm Re}\fii_2\big(\langle x,y\rangle\big)=0.
\end{align}
Since $S(\mathscr{A})$ is convex, $\lambda_o \fii_1+(1-\lambda_o)\fii_2\in S(\mathscr{A})$.
Put $\fii:=\lambda_o\fii_1+(1-\lambda_o)\fii_2$. We get then $\fii\in S(\mathscr{A})$ and $\fii(\langle x,x\rangle)=\|x\|^2$.
Hence $\varphi \in\Omega_x$. Also, by \eqref{two-states-rho-o}, ${\rm Re}\fii(\langle x,y\rangle)=0$.

Now we prove the implication (ii)$\Rightarrow$(i). Assume that there
exists $\varphi \in\Omega_x$ such that ${\rm Re}\fii(\langle x,y\rangle)=0$.
Fix arbitrarily $\alpha\in \mathbb{R}$. We have
\begin{align*}
\|x\|^2&=\fii\big(\langle x,x\rangle\big)={\rm Re}\fii\big(\langle x,x\rangle\big)+\alpha{\rm Re}\fii\big(\langle x,y\rangle\big)\\
&={\rm Re}\fii\big(\langle x,x+\alpha y\rangle\big)\leq\big|\fii\big(\langle x,x+\alpha y\rangle\big)\big|
\\&\leq\big\|\langle x,x+\alpha y\rangle\big\|\leq\|x\|\!\cdot\!\|x+\alpha y\|.
\end{align*}
Thus $\|x\|^2\leq\|x\|\!\cdot\!\|x+\alpha y\|$ and so $\|x\|\leq\|x+\alpha y\|$. Hence $x\perp^r_B y$.
\end{proof}
Finally, we are able to give a simple proof of \eqref{a-r-b-g--theo-bj-orth} using
the map $\rho_{_{+}}$, i.e. Theorem~\ref{T.3.2}.
Let us recall again: unlike \cite{A.R.1, B.G}, we did not apply a faithful
representation $\pi \colon \mathscr{A}\to \mathbb{B}(\mathcal{H})$ and linking algebra of $\mathscr{X}$.
Therefore the proofs of Theorem~\ref{T.3.2} and Theorem~\ref{T.877.2} are simpler
and shorter than the proofs in \cite{A.R.1, B.G}.
So, we now show that there is another (and easier) way to
get the celebrated result \eqref{a-r-b-g--theo-bj-orth}.
\begin{theorem}\label{T.877.2}
Let $\mathscr{X}$ be a Hilbert $\mathscr{A}$-module, and $x, y \in\mathscr{X}$.
The following statements are mutually equivalent:
\begin{itemize}
\item[(i)] $x\perp_B y$,
\item[(ii)] there exists $\varphi\in \Omega_x$ such that $\varphi(\langle x, y\rangle) = 0$,
\item[(iii)] $\|x + \lambda y\|^2\geq \|x\|^2 + |\lambda|^2 m(y)$ for all $\lambda\in\mathbb{C}$,
\end{itemize}
where $m(y): = \inf\big\{\varphi(\langle y, y\rangle): \, \varphi\in \mathcal{S}(\mathscr{A})\big\}$.
\end{theorem}
\begin{proof}
(i)$\Rightarrow$(ii) Let $x\perp_B y$. Then, by Theorem \ref{T.3.2} and \eqref{I.151.P.2.2}, we get
\begin{align}\label{I.1.T.877.2}
\displaystyle{\inf_{0\leq \theta < 2\pi}}\,\max\Big\{\mbox{Re}\,e^{i\theta}\varphi(\langle x, y\rangle): \, \varphi \in \Omega_x\Big\} \geq 0.
\end{align}
It is easy to see that the set $E:=\big\{\varphi(\langle x, y\rangle): \, \varphi \in \Omega_x\big\}$
is convex and hence its closure is a closed convex set. Therefore, by \eqref{I.1.T.877.2},
the set $E$ has such a position in the complex plane that it must contain
at least one value with positive real part, under all rotations around the origin. Thus
$E$ must contain zero, and so there is a $\varphi\in \Omega_x$ such that $\varphi(\langle x, y\rangle) = 0$.

(ii)$\Rightarrow$(iii) Suppose (ii) holds. Then, for every $\lambda \in \mathbb{C}$, we have
\begin{align*}
\|x + \lambda y\|^2 &\geq \varphi\big(\langle x + \lambda y, x + \lambda y\rangle\big)
\\& = \varphi(\langle x, x\rangle) + 2\mbox{Re}\,\big(\lambda \varphi(\langle x, y\rangle)\big) + |\lambda|^2\varphi(\langle y, y\rangle)
\\& = \|x\|^2 + |\lambda|^2\varphi(\langle y, y\rangle).
\end{align*}
Therefore
\begin{align*}
\|x + \lambda y\|^2 \geq \|x\|^2 + |\lambda|^2\varphi(\langle y, y\rangle),
\end{align*}
which yields $\|x + \lambda y\|^2\geq \|x\|^2 + |\lambda|^2 m(y)$ for all $\lambda\in\mathbb{C}$.

(iii)$\Rightarrow$(i) The implication is trivial.
\end{proof}
As a natural generalization of the notion of Birkhoff--James orthogonality,
the concept of strong Birkhoff--James orthogonality,
which involves modular structure of a Hilbert $C^*$-module was introduced in \cite{A.R.2}.
When $x$ and $y$ are elements of a Hilbert $\mathscr{A}$-module $\mathscr{X}$,
we consider the \textit{strong Birkhoff--James orthogonality}:
\begin{align*}
x\perp^s_B y\quad :\Leftrightarrow\quad \forall_{a\in \mathscr{A}}\ \|x\|\leq \|x +ya\|.
\end{align*}
One can easily observe that
$x\perp y\, \Longrightarrow \, x\perp^s_B y \, \Longrightarrow \, x\perp^r_B y$,
while the converses do not hold in general (see \cite{A.R.2}).

In the next result we establish characterizations of the strong Birkhoff--James orthogonality
for elements of a Hilbert $C^*$-module based on norm derivatives. We will apply our
new tools - Theorems~\ref{T.3.2}, \ref{C.4.2} and \ref{T.777.2}.
\begin{theorem}\label{T.6.2}
Let $\mathscr{X}$ be a Hilbert $\mathscr{A}$-module, and $x, y \in\mathscr{X}$.
The following statements are mutually equivalent:
\begin{itemize}
\item[(i)] $x\perp^s_B y$,
\item[(ii)] $\rho_{_{-}}(x, ya)\leq 0$ for all $a\in \mathscr{A}$,
\item[(iii)] $\rho_{_{+}}(x, ya)\geq 0$ for all $a\in \mathscr{A}$.
\end{itemize}
\end{theorem}
\begin{proof}
(i)$\Rightarrow$(ii) Let $x\perp^s_B y$. Then $x\perp^r_B y\langle y, x\rangle$. So, by Theorem \ref{T.777.2},
there exists $\varphi_o \in \Omega_x$ such
that ${\rm Re}\varphi_o\big(\big\langle x, y\langle y, x\rangle\big\rangle\big) = 0$.
Hence ${\rm Re}\varphi_o\big(\langle x, y\rangle\langle y, x\rangle\big) = 0$.
Thus $\varphi_o\big(\langle x, y\rangle\langle y, x\rangle\big) = 0$,
since $\varphi_o\big(\langle x, y\rangle\langle y, x\rangle\big)\in\mathbb{R}$.
Therefore,
\begin{align*}
\left|\mbox{Re}\,\varphi_o\big(\langle x, ya\rangle\big)\right|^2 \leq \left|\varphi_o\big(\langle x, y\rangle a\big)\right|^2 \stackrel{\eqref{inequality-positive-functional}}{\leq} \varphi_o\big(\langle x, y\rangle\langle y, x\rangle\big)\varphi_o(a^*a) = 0
\end{align*}
for each $a\in \mathscr{A}$.
This implies $\mbox{Re}\,\varphi_o\big(\langle x, ya\rangle\big) = 0$. Hence
\begin{align*}
\min\Big\{\mbox{Re}\,\varphi\big(\langle x, ya\rangle\big): \, \varphi \in \Omega_x\Big\} \leq 0
\end{align*}
and by Theorem \ref{C.4.2} it follows that $\rho_{_{-}}(x, ya)\leq 0$ for all $a\in \mathscr{A}$.

(ii)$\Rightarrow$(iii) By the condition (ii), $\rho_{_{-}}\big(x, y(-a)\big)\leq 0$ for all $a\in\mathscr{A}$.
Hence, by (P2), $\rho_{_{+}}(x, ya) = - \rho_{_{-}}\big(x, y(-a)\big) \geq 0$ for all $a\in\mathscr{A}$.

(iii)$\Rightarrow$(i) Suppose (iii) holds.
We may assume that $x\neq0$ otherwise (i) trivially holds. So, fix $a\in\mathscr{A}$.
By Theorem \ref{T.3.2}, there exists a state $\varphi_a$ in $\Omega_x$ such that
$\mbox{Re}\,\varphi_a\big(\langle x, ya\rangle\big) \geq 0$. Then we have
\begin{align*}
0 &\leq \mbox{Re}\,\varphi_a\big(\langle x, ya\rangle\big)
= \mbox{Re}\,\varphi_a\big(\langle x, x + ya\rangle\big) - \mbox{Re}\,\varphi_a\big(\langle x, x\rangle\big)
\\&\leq \big|\varphi_a\big(\langle x, x + ya\rangle\big)\big| - \|x\|^2 \leq \big\|\langle x, x + ya\rangle\big\| - \|x\|^2
\\&\leq \|x\|\cdot\|x + ya\| - \|x\|^2 = \|x\|\cdot(\|x + ya\| - \|x\|).
\end{align*}
This implies that $\|x\|\leq\|x + ya\|$. Hence $x\perp^s_B y$.
\end{proof}
We recall that (see \cite{Dra, Mil}) two elements $x$ and $y$ of a normed linear space $X$ are $\rho$-orthogonal
if $\rho(x, y) : = \frac{\rho_{_{+}}(x, y) + \rho_{_{-}}(x, y)}{2} = 0$, and in this case we write $x\perp_{\rho} y$.
It is worth mentioning that the notion of $\rho$-orthogonality may be a strong
tool. Indeed, the open problem posed in \cite{A.S.T} was solved in in the paper \cite{Wo.2019-laa} and
the concept of $\rho$-orthogonality played a significant role.
For facts about the $\rho$-orthogonality in normed linear spaces, we refer
the reader to \cite{C.W.1, C.W.2, M.Z.D, Z.M}.

If $x$ and $y$ are elements of a Hilbert $\mathscr{A}$-module $\mathscr{X}$, then we have
\begin{align}\label{I.1.R.1}
x\perp y\, \Longrightarrow \, x\perp_{\rho} y \, \Longrightarrow \, x\perp^r_B y.
\end{align}
Indeed, if $x\perp y$, then $\langle x, y\rangle = 0$. Thus, for
every $\varphi \in \Omega_x$, we have $\mbox{Re}\,\varphi(\langle x, y\rangle) = 0$
and so by Theorem~\ref{T.3.2} and Theorem~\ref{C.4.2}
we obtain $\rho_{_{-}}(x, y) = \rho_{_{+}}(x, y) =0$. Hence $\rho(x, y) = 0$, and thus $x\perp_{\rho} y$.

Further, if $x\perp_{\rho} y$, then $\rho_{_{-}}(x, y) + \rho_{_{+}}(x, y) =0$ and
therefore $\rho_{_{-}}(x, y)\leq 0\leq\rho_{_{+}}(x, y)$ by (P1).
It follows \eqref{I.150.P.2.2} from that $x\perp^r_B y$.

As an immediate consequence of Theorem~\ref{T.3.2} and Theorem \ref{C.4.2} we obtain a characterization of
the $\rho$-orthogonality in Hilbert $C^*$-modules as follows.
\begin{theorem}\label{T.60.2}
Let $\mathscr{X}$ be a Hilbert $\mathscr{A}$-module, and $x, y \in\mathscr{X}$.
The following statements are equivalent:
\begin{itemize}
\item[(i)] $x\perp_{\rho} y$,
\item[(ii)] $\max\big\{{\rm Re}\,\varphi(\langle x, y\rangle): \, \varphi \in \Omega_x\big\}
= \max\big\{-{\rm Re}\,\varphi(\langle x, y\rangle): \, \varphi \in \Omega_x \big\}$.
\end{itemize}
\end{theorem}
\begin{remark}
Notice that the converses in \eqref{I.1.R.1} do not hold in general.
For example, consider $\mathbb{M}_{2}(\mathbb{C})$ as a Hilbert $\mathbb{M}_{2}(\mathbb{C})$-module
and let $T = \begin{bmatrix} 1 & 0 \\ 0 & 1 \end{bmatrix}$,
$S = \begin{bmatrix} -1 & 0 \\ 0 & 1 \end{bmatrix}$ and $R = \begin{bmatrix} -1 & 0 \\ 0 & 0 \end{bmatrix}$.
Then simple computations show that
\begin{align*}
\rho_{_{+}}(T, S) = -\rho_{_{-}}(T, S) = - \rho_{_{-}}(T, R) = 1 \quad \mbox{and} \quad \rho_{_{+}}(T, R) = 0.
\end{align*}
Hence $T\perp_{\rho} S$. But $T\not\perp S$, since $\langle T, S\rangle = S \neq 0$.
Also, $T\not\perp_{\rho} R$ but
\begin{align*}
\|T + \alpha R\| = \left\|\begin{bmatrix}
1- \alpha & 0
\\ 0 & 1
\end{bmatrix}\right\| = \max\{|1 - \alpha|, 1\} \geq 1 = \|T\|
\end{align*}
for all $\alpha\in\mathbb{R}$. Therefore $T\perp^r_B R$.

To end the work we show that the orthogonalities $\perp_{\rho}$ and $\perp^s_B$ are incomparable.
Indeed, since for $C = \begin{bmatrix}
1 & 0
\\ 0 & -1
\end{bmatrix}$ we have
$\|T + SC\| = 0 < 1 = \|T\|$, we get
$T\not\perp^s_B S$. Furthermore, for every $A = \begin{bmatrix}
a & b
\\ c & d
\end{bmatrix}\in \mathbb{M}_{2}(\mathbb{C})$ we have
$\|T + RA\| = \left\|\begin{bmatrix}
1- a & -b
\\ 0 & 1
\end{bmatrix}\right\| \geq 1 = \|T\|$,
whence $T\perp^s_B R$. Therefore $\perp_{\rho}\nsubseteq\perp^s_B$ and $\perp_{\rho}\nsupseteq\perp^s_B$.
\end{remark}
\bibliographystyle{amsplain}

\end{document}